\documentclass[10pt,reqno]{amsart} 
\usepackage[a4paper, margin=2.3cm]{geometry}

\usepackage{amsthm}

\usepackage{bbm}
\usepackage{mathrsfs}
\usepackage{graphicx}

\usepackage[utf8]{inputenc}
\usepackage{xypic}

\usepackage{multirow}

\usepackage{enumitem}

\usepackage{float}
\usepackage{amsbsy}
\usepackage[all]{xy}\usepackage{xypic}
\usepackage{braket}
\usepackage[colorlinks = true,citecolor=black]{hyperref}
\usepackage{dynkin-diagrams}

\hypersetup{
     colorlinks=true,
     linkcolor=blue,
     filecolor=blue,
     citecolor = blue,      
     urlcolor=cyan,
     }

\newtheorem{theorem}{Theorem}[section]
\newtheorem*{theorem*}{Theorem}
\newtheorem{theorem-non}{Theorem}
\newtheorem{lemma-non}{Lemma}

\theoremstyle{definition} 
\newtheorem{thm}{Theorem}

\theoremstyle{definition} 
\newtheorem{corollarynon}{Corollary}

\newtheorem{conjecture-non}{Conjecture}

\newtheorem{corollary-non}{Corollary}
\newtheorem{proposition}[theorem]{Proposition}

\newtheorem*{lemma*}{Lemma}
\newtheorem{corollary}[theorem]{Corollary}

\newtheorem*{conjecture*}{Conjecture}

\theoremstyle{definition}

\theoremstyle{remark}
\newtheorem{remark}[theorem]{Remark}

\DeclareMathOperator{\Diag}{Diag}

\DeclareMathOperator{\rank}{rank}
\DeclareMathOperator{\htt}{ht}
\DeclareMathOperator{\Spin}{Spin}

\numberwithin{equation}{section}

\title{Twisted K\"ahler-Einstein metrics on flag varieties}
\author{Eder M. Correa}
\author{Lino Grama}
\address{University of Campinas (UNICAMP), Institute of Mathematics, Statistics and Scientific Computing (IMECC), Campinas, Brazil}
\email{ederc@unicamp.br, lgrama@unicamp.br}
\date{\today}

\begin{document}
\begin{abstract}
    In this paper, we describe invariant twisted K\"ahler-Einstein (tKE) metrics on flag varieties. We also explore some applications of the ideas involved in the proof of our main result to the existence of invariant twisted constant scalar curvature K\"{a}hler metrics. Also, we provide a precise description for the greatest Ricci lower bound of an arbitrary K\"{a}hler class on a flag variety. By means of this description, we establish some inequalities related to optimal volume upper bounds for K\"{a}hler metrics just using tools from Lie theory. Further, we describe the set of tKE metrics for several examples, including full flag varieties, the projectivization of the tangent bundle of $\mathbbm{P}^{n+1}$, and families of flag varieties with Picard number $2$. 
\end{abstract}
\maketitle
\section{Introduction}

Let $X$ be a compact K\"ahler manifold and let $\xi$ be a K\"ahler class on $X$. Fixed a smooth $(1,1)$-form $\beta \in 2\pi  \big (c_1(X) - \xi \big )$, the solutions of the equation 
\begin{equation}\label{TKE-equation}
    {\rm{Ric}}(\omega)=\omega + \beta, 
\end{equation}
are called {\em twisted K\"ahler-Einstein} (tKE) metrics. As in the usual Fano case ($\xi = c_{1}(X)$ and $\beta = 0)$, we have that Eq. (\ref{TKE-equation}) is not always solvable. In the particular case that $\xi = c_{1}(L)$ and $\beta > 0$, the existence of solutions of Eq. (\ref{TKE-equation}) was characterized in \cite{BBSJ} using a twisted analogue of the $\delta$-{\em invariant} originally defined in \cite{FujitaOdaka}. As shown in \cite{BLZ}, \cite{BBSJ}, \cite{BLZ2}, the $\delta$-invariant is the right threshold to detect Ding-stability, an algebraic notion designed for the existence of tKE metrics. However, $\delta$-invariant is not easy to compute in general, see for instance \cite{BlumJonsson}. Further, tKE metrics also appear, for instance, in the study of K\"ahler-Ricci flow through singularities, see \cite{song-tian}, \cite{song-tian2}. There is also a twisted version of the K\"{a}hler-Ricci flow which was studied in \cite{JLiu}, \cite{JLiuYWang}, \cite{collins1}. In this paper, we restrict ourselves to the study of Eq. (\ref{TKE-equation}) in the setting of generalized flag varieties. A flag variety can be described as a quotient $X_{P} = G^{\mathbbm{C}}/P$, where $G^{\mathbbm{C}}$ is a semisimple complex algebraic group and $P$ is a parabolic subgroup (Borel-Remmert \cite{BorelRemmert}). Regarding $G^{\mathbbm{C}}$ as a complex analytic space, without loss of generality, we may assume that $G^{\mathbbm{C}}$ is a connected simply connected complex simple Lie group. Fixed a compact real form $G \subset G^{\mathbbm{C}}$, the main purpose in this work is to characterize the existence of $G$-invariant tKE metrics on flag varieties using essentially tools of Lie theory. As in the case of $G$-invariant K\"{a}hler-Einstein metrics, the main philosophy is to reduce the underlying non-linear PDE problem provided by Eq. (\ref{TKE-equation}) to an algebraic problem involving elements of the theory of semisimple Lie groups and Lie algebras. The advantage of this approach is that it allows us to describe $G$-invariant tKE metrics in a precise and explicit way. We also explore some applications of the ideas involved in the proof of our main result to the existence of certain invariant twisted constant scalar curvature K\"{a}hler metrics. Further, we provide a precise description for the {\textit{greatest Ricci lower bound}} of every K\"{a}hler class on a flag variety. By means of this description, we establish some inequalities related to optimal volume upper bounds for K\"{a}hler metrics just using tools from Lie theory. Additionally, we describe the set of tKE metrics for several examples, including full flag varieties, the projectivization of the tangent bundle of $\mathbbm{P}^{n+1}$, and families of flag varieties with Picard number $2$.

\subsection{Main results} Let $X_{P} = G^{\mathbbm{C}}/P$ be a complex flag variety. Considering ${\text{Lie}}(G^{\mathbbm{C}}) = \mathfrak{g}^{\mathbbm{C}}$, if we choose a Cartan subalgebra $\mathfrak{h} \subset \mathfrak{g}^{\mathbbm{C}}$, and a simple root system $\Sigma \subset \mathfrak{h}^{\ast}$, up to conjugation, we have that $P = P_{\Theta}$, for some $\Theta \subset \Sigma$, where $P_{\Theta}$ is a parabolic Lie subgroup determined by $\Theta$, see for instance \cite{Akhiezer}. The main result which we prove in this work is the following 
\begin{thm}
Let $L \in {\rm{Pic}}(X_{P})$ and let $\beta \in c_{1}(L)$ be a $G$-invariant $(1,1)$-form. Then there exists a (unique) $G$-invariant K\"{a}hler metric $\omega$ on $X_{P}$, satisfying 
\begin{equation}
{\rm{Ric}}(\omega) = \omega + \beta,
\end{equation}
if, and only if,
\begin{equation}
\int_{\mathbbm{P}_{\alpha}^{1}}\beta < 2\pi\langle \delta_{P},h_{\alpha}^{\vee}  \rangle,
\end{equation}
$\forall \alpha \in \Sigma \backslash \Theta$, where $\mathbbm{P}_{\alpha}^{1} \subset X_{P}$, $\alpha \in \Sigma \backslash \Theta$, are generators of the cone of curves ${\rm{NE}}(X_{P})$. 
\end{thm}

From the above result, one can prove that Eq. (\ref{TKE-equation}) can be always solved in the setting of flag varieties if $\beta$ is $G$-invariant, i.e., we have the following.

\begin{corollarynon}
Given a K\"{a}hler class $\xi$ on $X_{P}$ and a $G$-invariant $(1,1)$-form $ \beta \in 2\pi \big ( c_{1}(X_{P}) - \xi \big)$, then there exist a unique $G$-invariant K\"{a}hler metric $\omega \in 2\pi \xi$, such that 
\begin{equation}
\label{tKEeq}
{\rm{Ric}}(\omega) = \omega + \beta.
\end{equation}
\end{corollarynon}
Given a semipositive $(1,1)$-form $\beta$ on a compact K\"{a}hler manifold $(X,\omega)$, denoting by $S(\omega)$ the Chern scalar curvature of $\omega$, if 
\begin{equation}
S(\omega) - \Lambda_{\omega}(\beta) = {\text{const.}},
\end{equation}
we say that $\omega$ is a $\beta${\textit{-twisted constant scalar curvature K\"{a}hler metric}} ($\beta$-twisted cscK metric). The equation above arises in the study and construction of constant scalar curvature K\"{a}hler metrics (cscK metrics), e.g. \cite{Fine1}, \cite{Fine2}, \cite{SongTian}. In \cite[Theorem 4.5]{BB}, it was shown that a $\beta$-twisted cscK metric is unique in each cohomology class. In the particular setting of flag varieties, by taking the trace with respect to $\omega$ in Eq. (\ref{tKEeq}), we obtain the following result.

\begin{corollarynon}
Given a K\"{a}hler class $\xi$ on $X_{P}$ and a $G$-invariant $(1,1)$-form $ \beta \in 2\pi \big ( c_{1}(X_{P}) - \xi \big)$, then there exists a (unique) $G$-invariant K\"{a}hler metric $\omega \in 2\pi \xi$ with constant $\beta$-twisted scalar curvature, such that
\begin{equation}
S(\omega) - \Lambda_{\omega}(\beta) = \dim_{\mathbbm{C}}(X_{P}),
\end{equation}
where $S(\omega)$ denotes the Chern scalar curvature of $\omega$.
\end{corollarynon}

Let $\mathcal{K}(X)$ be the K\"{a}hler cone of a compact K\"{a}hler manifold $X$. For any K\"{a}hler class $\xi \in \mathcal{K}(X)$, one can define its greatest Ricci lower bound $R(\xi)$ as being
\begin{equation}
R(\xi) := \sup\{ r \in \mathbbm{R} \ | \ \exists \ {\text{K\"{a}hler form}} \ \omega \in 2\pi \xi, \ {\text{s.t.}} \ {\rm{Ric}}(\omega) \geq r \omega  \}. 
\end{equation}
This invariant was first studied by Tian in \cite{Tian1} for the case $\xi = c_{1}(X)$ and further studied, for instance, in \cite{YRubinstein}, \cite{YRubinstein1}, \cite{SzeGabor}, \cite{SongWang}. As in the case of the $\delta$-invariant, few examples of explicit computations of the greatest Ricci lower bound are known so far, see for instance \cite{ChiLi} for case when $X$ is a toric Fano manifold. Using the previous results, we prove the following explicit expression for the greatest Ricci lower bound $R(\xi)$ of any K\"{a}hler class $\xi$ on a flag variety $X_{P}$.
\begin{corollarynon}
\label{GRLBFlag}
Let $R \colon \mathcal{K}(X_{P}) \to \mathbbm{R}$, such that $R(\xi)$ is the greatest Ricci lower bound of $\xi \in \mathcal{K}(X_{P})$. Then, we have
\begin{equation}
R(\xi) = \min \bigg \{ \frac{\langle \delta_{P},h_{\alpha}^{\vee}  \rangle}{a_{\alpha}} \ \bigg | \  \alpha \in \Sigma \backslash \Theta \bigg \},
\end{equation}
such that $a_{\alpha} = \langle \xi, [\mathbbm{P}_{\alpha}^{1}]\rangle$, $\forall \alpha \in \Sigma \backslash \Theta$.
\end{corollarynon}
Given a K\"{a}hler class $\xi \in \mathcal{K}(X)$ of a Fano manifold $X$, it was shown in \cite[Theorem 4.1]{ZhangKewei} that 
\begin{equation}
R(\xi)^{n}{\rm{Vol}}(\xi) \leq (n+1)^{n},
\end{equation}
such that $n = \dim_{\mathbbm{C}}(X)$. Also, if $X$ is Fano and admits a K\"{a}hler-Einstein metric, it follows from \cite{Fujita} that
\begin{equation}
\label{VolFano}
(-K_{X})^{n} \leq (n+1)^{n}.
\end{equation}
It is worth pointing out that the inequality above was first proved for the particular case of flag varieties in \cite{Snow}. From the explicit description provided in Corollary \ref{GRLBFlag} for the the greatest Ricci lower bound of K\"{a}hler classes of flag varieties, we prove the following result relating the inequalities above.
\begin{thm}
\label{VolBound}
For every $\xi \in \mathcal{K}(X_{P})$, the following inequalities hold
\begin{equation}
R(\xi)^{n}{\rm{Vol}}(\xi) \leq (-K_{X_{P}})^{n} \leq (n+1)^{n},
\end{equation}
such that $n = \dim_{\mathbbm{C}}(X_{P})$.
\end{thm}

It is worth mentioning that the proof which we present in this work for the result above is independent of the results provided in \cite[Theorem 4.1]{ZhangKewei} and \cite{Fujita}. In fact, we prove Theorem \ref{VolBound} using essentially tools from Lie theory and the result provided in \cite{Snow}. 

\subsection{Outline of the paper} In Section \ref{generalities}, we review some facts about the geometry of flag varieties. In Section \ref{main-results}, we prove our main results. In Section \ref{sec-examples}, we work out some examples in order to determine twisted K\"ahler-Einstein metrics. These examples include full flag varieties, projectivization of the tangent bundle of $\mathbbm{P}^{n+1}$, and families of flag varieties with Picard number $2$.
 
\section{Generalities on flag varieties}\label{generalities}
In this section, we review some basic facts about flag varieties. For more details on the subject presented in this section, we suggest \cite{Akhiezer}, \cite{Flagvarieties}, \cite{HumphreysLAG}, \cite{BorelRemmert}.
\subsection{The Picard group of flag varieties}
\label{subsec3.1}
Let $G^{\mathbbm{C}}$ be a connected, simply connected, and complex Lie group with simple Lie algebra $\mathfrak{g}^{\mathbbm{C}}$. By fixing a Cartan subalgebra $\mathfrak{h}$ and a simple root system $\Sigma \subset \mathfrak{h}^{\ast}$, we have a decomposition of $\mathfrak{g}^{\mathbbm{C}}$ given by
\begin{center}
$\mathfrak{g}^{\mathbbm{C}} = \mathfrak{n}^{-} \oplus \mathfrak{h} \oplus \mathfrak{n}^{+}$, 
\end{center}
where $\mathfrak{n}^{-} = \sum_{\alpha \in \Pi^{-}}\mathfrak{g}_{\alpha}$ and $\mathfrak{n}^{+} = \sum_{\alpha \in \Pi^{+}}\mathfrak{g}_{\alpha}$, here we denote by $\Pi = \Pi^{+} \cup \Pi^{-}$ the root system associated to the simple root system $\Sigma = \{\alpha_{1},\ldots,\alpha_{m}\} \subset \mathfrak{h}^{\ast}$. Let us denote by $\kappa$ the Cartan-Killing form of $\mathfrak{g}^{\mathbbm{C}}$. From this, for every  $\alpha \in \Pi^{+}$ we have $h_{\alpha} \in \mathfrak{h}$, such  that $\alpha = \kappa(\cdot,h_{\alpha})$, and we can choose $x_{\alpha} \in \mathfrak{g}_{\alpha}$ and $y_{-\alpha} \in \mathfrak{g}_{-\alpha}$, such that $[x_{\alpha},y_{-\alpha}] = h_{\alpha}$. From these data, we can define a Borel subalgebra by setting $\mathfrak{b} = \mathfrak{h} \oplus \mathfrak{n}^{+}$. Now we consider the following result (see for instance \cite{Flagvarieties}, \cite{HumphreysLAG}):
\begin{theorem}
Any two Borel subgroups are conjugate.
\end{theorem}
From the result above, given a Borel subgroup $B \subset G^{\mathbbm{C}}$, up to conjugation, we can always suppose that $B = \exp(\mathfrak{b})$. In this setting, given a parabolic Lie subgroup $P \subset G^{\mathbbm{C}}$, without loss of generality, we can suppose that
\begin{center}
$P  = P_{\Theta}$, \ for some \ $\Theta \subseteq \Sigma$,
\end{center}
where $P_{\Theta} \subset G^{\mathbbm{C}}$ is the parabolic subgroup which integrates the Lie subalgebra 
\begin{center}

$\mathfrak{p}_{\Theta} = \mathfrak{n}^{+} \oplus \mathfrak{h} \oplus \mathfrak{n}(\Theta)^{-}$, \ with \ $\mathfrak{n}(\Theta)^{-} = \displaystyle \sum_{\alpha \in \langle \Theta \rangle^{-}} \mathfrak{g}_{\alpha}$. 

\end{center}
By definition, it is straightforward to show that $P_{\Theta} = N_{G^{\mathbbm{C}}}(\mathfrak{p}_{\Theta})$, where $N_{G^{\mathbbm{C}}}(\mathfrak{p}_{\Theta})$ is the normalizer in  $G^{\mathbbm{C}}$ of $\mathfrak{p}_{\Theta} \subset \mathfrak{g}^{\mathbbm{C}}$. In what follows, it will be useful for us to consider the following basic chain of Lie subgroups

\begin{center}

$T^{\mathbbm{C}} \subset B \subset P \subset G^{\mathbbm{C}}$.

\end{center}
For each element in the aforementioned chain of Lie subgroups we have the following characterization: 

\begin{itemize}

\item $T^{\mathbbm{C}} = \exp(\mathfrak{h})$;  \ \ (complex torus)

\item $B = N^{+}T^{\mathbbm{C}}$, where $N^{+} = \exp(\mathfrak{n}^{+})$; \ \ (Borel subgroup)

\item $P = P_{\Theta} = N_{G^{\mathbbm{C}}}(\mathfrak{p}_{\Theta})$, for some $\Theta \subset \Sigma \subset \mathfrak{h}^{\ast}$. \ \ (parabolic subgroup)

\end{itemize}
Now let us recall some basic facts about the representation theory of $\mathfrak{g}^{\mathbbm{C}}$, more details can be found in \cite{Humphreys}. For every $\alpha \in \Sigma$, we set 
$$h_{\alpha}^{\vee} = \frac{2}{\kappa(h_{\alpha},h_{\alpha})}h_{\alpha}.$$ 
The fundamental weights $\{\varpi_{\alpha} \ | \ \alpha \in \Sigma\} \subset \mathfrak{h}^{\ast}$ of $(\mathfrak{g}^{\mathbbm{C}},\mathfrak{h})$ are defined by requiring that $\varpi_{\alpha}(h_{\beta}^{\vee}) = \delta_{\alpha \beta}$, $\forall \alpha, \beta \in \Sigma$. We denote by 
$$\Lambda^{+} = \bigoplus_{\alpha \in \Sigma}\mathbbm{Z}_{\geq 0}\varpi_{\alpha},$$ 
the set of integral dominant weights of $\mathfrak{g}^{\mathbbm{C}}$. Let $V$ be an arbitrary finite dimensional $\mathfrak{g}^{\mathbbm{C}}$-module. By considering its weight space decomposition
\begin{center}
$\displaystyle{V = \bigoplus_{\mu \in \Pi(V)}V_{\mu}},$ \ \ \ \ 
\end{center}
such that $V_{\mu} = \{v \in V \ | \ h \cdot v = \mu(h)v, \ \forall h \in \mathfrak{h}\} \neq \{0\}$, $\forall \mu \in \Pi(V) \subset \mathfrak{h}^{\ast}$, from the Lie algebra representation theory we have the following facts:
\begin{enumerate}
\item A highest weight vector (of weight $\lambda$) in a $\mathfrak{g}^{\mathbbm{C}}$-module $V$ is a non-zero vector $v_{\lambda}^{+} \in V_{\lambda}$, such that 
\begin{center}
$x \cdot v_{\lambda}^{+} = 0$, \ \ \ \ \ ($\forall x \in \mathfrak{n}^{+}$).
\end{center}
Such a $\lambda \in \Pi(V)$ satisfying the above condition is called highest weight of $V$;
\item $V$ irreducible $\Longrightarrow$ $\exists$ highest weight vector $v_{\lambda}^{+} \in V$ (unique up to non-zero
scalar multiples) for some $\lambda \in \Pi(V)$; 
\item If $\lambda \in \Lambda^{+}$, then there exists a finite dimensional irreducible $\mathfrak{g}^{\mathbbm{C}}$-module $V$ which has $\lambda$ as highest weight. In this case, we denote $V = V(\lambda)$;

\item For all $\lambda \in \Lambda^{+}$, we have $V(\lambda) = \mathfrak{U}(\mathfrak{g}^{\mathbbm{C}}) \cdot v_{\lambda}^{+}$, where $\mathfrak{U}(\mathfrak{g}^{\mathbbm{C}})$ is the universal enveloping algebra of $\mathfrak{g}^{\mathbbm{C}}$;
\item The fundamental representations are defined by $V(\varpi_{\alpha})$, $\alpha \in \Sigma$; 

\item Given $\lambda \in \Lambda^{+}$, such that $\lambda = \sum_{\alpha}n_{\alpha}\varpi_{\alpha}$, we have
\begin{center}
$\displaystyle v_{\lambda}^{+} = \bigotimes_{\alpha \in \Sigma}(v_{\varpi_{\alpha}}^{+})^{\otimes n_{\alpha}}$ \ \ \ and \ \ \ $\displaystyle V(\lambda) = \mathfrak{U}(\mathfrak{g}^{\mathbbm{C}}) \cdot v_{\lambda}^{+} \subset \bigotimes_{\alpha \in \Sigma}V(\varpi_{\alpha})^{\otimes n_{\alpha}};$
\end{center}
\item For all $\lambda \in \Lambda^{+}$, we have the following equivalence of induced irreducible representations
\begin{center}
$\varrho \colon G^{\mathbbm{C}} \to {\rm{GL}}(V(\lambda))$ \ $\Longleftrightarrow$ \ $\varrho_{\ast} \colon \mathfrak{g}^{\mathbbm{C}} \to \mathfrak{gl}(V(\lambda))$,
\end{center}
such that $\varrho(\exp(x)) = \exp(\varrho_{\ast}x)$, $\forall x \in \mathfrak{g}^{\mathbbm{C}}$, notice that $G^{\mathbbm{C}} = \langle \exp(\mathfrak{g}^{\mathbbm{C}}) \rangle$.
\end{enumerate}
In what follows, for any representation $\varrho \colon G^{\mathbbm{C}} \to {\rm{GL}}(V(\lambda))$, for the sake of simplicity, we shall denote $\varrho(g)v = gv$, for all $g \in G^{\mathbbm{C}}$, and all $v \in V(\lambda)$. Let $G \subset G^{\mathbbm{C}}$ be a compact real form for $G^{\mathbbm{C}}$. Given a complex flag variety $X_{P} = G^{\mathbbm{C}}/P$, regarding $X_{P}$ as a homogeneous $G$-space, that is, $X_{P} = G/G\cap P$, the following theorem allows us to describe all $G$-invariant K\"{a}hler structures on $X_{P}$ by means of elements of representation theory.
\begin{theorem}[Azad-Biswas, \cite{AZAD}]
\label{AZADBISWAS}
Let $\omega \in \Omega^{1,1}(X_{P})^{G}$ be a closed invariant real $(1,1)$-form, then we have

\begin{center}

$\pi^{\ast}\omega = \sqrt{-1}\partial \overline{\partial}\varphi$,

\end{center}
where $\pi \colon G^{\mathbbm{C}} \to X_{P}$, and $\varphi \colon G^{\mathbbm{C}} \to \mathbbm{R}$ is given by 
\begin{center}
$\varphi(g) = \displaystyle \sum_{\alpha \in \Sigma \backslash \Theta}c_{\alpha}\log \big (||gv_{\varpi_{\alpha}}^{+}|| \big )$, \ \ \ \ $(\forall g \in G^\mathbbm{C})$
\end{center}
with $c_{\alpha} \in \mathbbm{R}$, $\forall \alpha \in \Sigma \backslash \Theta$. Conversely, every function $\varphi$ as above defines a closed invariant real $(1,1)$-form $\omega_{\varphi} \in \Omega^{1,1}(X_{P})^{G}$. Moreover, $\omega_{\varphi}$ defines a $G$-invariant K\"{a}hler form on $X_{P}$ if, and only if, $c_{\alpha} > 0$,  $\forall \alpha \in \Sigma \backslash \Theta$.
\end{theorem}

\begin{remark}
\label{innerproduct}
It is worth pointing out that the norm $|| \cdot ||$ in the last theorem is a norm induced from some fixed $G$-invariant inner product $\langle \cdot, \cdot \rangle_{\alpha}$ on $V(\varpi_{\alpha})$, for every $\alpha \in \Sigma \backslash \Theta$. 
\end{remark}

\begin{remark}
An important consequence of Theorem \ref{AZADBISWAS} is that it allows us to describe the local K\"{a}hler potential for any homogeneous K\"{a}hler metric in a quite concrete way, for some examples of explicit computations we suggest \cite{CorreaGrama}, \cite{Correa}.
\end{remark}

By means of the above theorem we can describe the unique $G$-invariant representative in each integral class in $H^{2}(X_{P},\mathbbm{Z})$. In fact, consider the associated $P$-principal bundle $P \hookrightarrow G^{\mathbbm{C}} \to X_{P}$. By choosing a trivializing open covering $X_{P} = \bigcup_{i \in I}U_{i}$, in terms of $\check{C}$ech cocycles we can write 
\begin{center}
$G^{\mathbbm{C}} = \Big \{(U_{i})_{i \in I}, \psi_{ij} \colon U_{i} \cap U_{j} \to P \Big \}$.
\end{center}
Given a fundamental weight $\varpi_{\alpha} \in \Lambda^{+}$, we consider the induced character $\chi_{\varpi_{\alpha}} \in {\text{Hom}}(T^{\mathbbm{C}},\mathbbm{C}^{\times})$, such that $(d\chi_{\varpi_{\alpha}})_{e} = \varpi_{\alpha}$. From the homomorphism $\chi_{\varpi_{\alpha}} \colon P \to \mathbbm{C}^{\times}$ one can equip $\mathbbm{C}$ with a structure of $P$-space, such that $pz = \chi_{\varpi_{\alpha}}(p)^{-1}z$, $\forall p \in P$, and $\forall z \in \mathbbm{C}$. Denoting by $\mathbbm{C}_{-\varpi_{\alpha}}$ this $P$-space, we can form an associated holomorphic line bundle $\mathscr{O}_{\alpha}(1) = G^{\mathbbm{C}} \times_{P}\mathbbm{C}_{-\varpi_{\alpha}}$, which can be described in terms of $\check{C}$ech cocycles by
\begin{equation}
\label{linecocycle}
\mathscr{O}_{\alpha}(1) = \Big \{(U_{i})_{i \in I},\chi_{\varpi_{\alpha}}^{-1} \circ \psi_{i j} \colon U_{i} \cap U_{j} \to \mathbbm{C}^{\times} \Big \},
\end{equation}
that is, $\mathscr{O}_{\alpha}(1) = \{g_{ij}\} \in \check{H}^{1}(X_{P},\mathcal{O}_{X_{P}}^{\ast})$, such that $g_{ij} = \chi_{\varpi_{\alpha}}^{-1} \circ \psi_{i j}$, for every $i,j \in I$. 
\begin{remark}
\label{parabolicdec}
We observe that, if we have a parabolic Lie subgroup $P \subset G^{\mathbbm{C}}$, such that $P = P_{\Theta}$, the decomposition 
\begin{equation}
P_{\Theta} = \big[P_{\Theta},P_{\Theta} \big]T(\Sigma \backslash \Theta)^{\mathbbm{C}}, \ \  {\text{such that }} \ \ T(\Sigma \backslash \Theta)^{\mathbbm{C}} = \exp \Big \{ \displaystyle \sum_{\alpha \in  \Sigma \backslash \Theta}a_{\alpha}h_{\alpha} \ \Big | \ a_{\alpha} \in \mathbbm{C} \Big \},
\end{equation}
e.g. \cite[Proposition 8]{Akhiezer}, shows us that ${\text{Hom}}(P,\mathbbm{C}^{\times}) = {\text{Hom}}(T(\Sigma \backslash \Theta)^{\mathbbm{C}},\mathbbm{C}^{\times})$. Therefore, if we take $\varpi_{\alpha} \in \Lambda^{+}$, such that $\alpha \in \Theta$, it follows that $\mathscr{O}_{\alpha}(1) = X_{P} \times \mathbbm{C}$, i.e., the associated holomorphic line bundle $\mathscr{O}_{\alpha}(1)$ is trivial.
\end{remark}

Given $\mathscr{O}_{\alpha}(1) \in {\text{Pic}}(X_{P})$, such that $\alpha \in \Sigma \backslash \Theta$, as described above, if we consider an open covering $X_{P} = \bigcup_{i \in I} U_{i}$ which trivializes both $P \hookrightarrow G^{\mathbbm{C}} \to X_{P}$ and $ \mathscr{O}_{\alpha}(1) \to X_{P}$, by taking a collection of local sections $(s_{i})_{i \in I}$, such that $s_{i} \colon U_{i} \to G^{\mathbbm{C}}$, we can define $q_{i} \colon U_{i} \to \mathbbm{R}^{+}$, such that 
\begin{equation}
\label{functionshermitian}
q_{i} =  {\mathrm{e}}^{-2\pi \varphi_{\varpi_{\alpha}} \circ s_{i}} = \frac{1}{||s_{i}v_{\varpi_{\alpha}}^{+}||^{2}},
\end{equation}
for every $i \in I$. Since $s_{j} = s_{i}\psi_{ij}$ on $U_{i} \cap U_{j} \neq \emptyset$, and $pv_{\varpi_{\alpha}}^{+} = \chi_{\varpi_{\alpha}}(p)v_{\varpi_{\alpha}}^{+}$, for every $p \in P$, such that $\alpha \in \Sigma \backslash \Theta$, the collection of functions $(q_{i})_{i \in I}$ satisfy $q_{j} = |\chi_{\varpi_{\alpha}}^{-1} \circ \psi_{ij}|^{2}q_{i}$ on $U_{i} \cap U_{j} \neq \emptyset$. Hence, we obtain a collection of functions $(q_{i})_{i \in I}$ which satisfies on $U_{i} \cap U_{j} \neq \emptyset$ the following relation
\begin{equation}
\label{collectionofequ}
q_{j} = |g_{ij}|^{2}q_{i},
\end{equation}
such that $g_{ij} = \chi_{\varpi_{\alpha}}^{-1} \circ \psi_{i j}$, where $i,j \in I$. From this, we can define a Hermitian structure $H$ on $\mathscr{O}_{\alpha}(1)$ by taking on each trivialization $f_{i} \colon L_{\chi_{\varpi_{\alpha}}} \to U_{i} \times \mathbbm{C}$ a metric defined by
\begin{equation}
\label{hermitian}
H(f_{i}^{-1}(x,v),f_{i}^{-1}(x,w)) = q_{i}(x) v\overline{w},
\end{equation}
for $(x,v),(x,w) \in U_{i} \times \mathbbm{C}$. The Hermitian metric above induces a Chern connection $\nabla = d + \partial \log H$ with curvature $F_{\nabla}$ satisfying (locally)
\begin{equation}
\displaystyle \frac{\sqrt{-1}}{2\pi}F_{\nabla} \Big |_{U_{i}} = \frac{\sqrt{-1}}{2\pi} \partial \overline{\partial}\log \Big ( \big | \big | s_{i}v_{\varpi_{\alpha}}^{+}\big | \big |^{2} \Big).
\end{equation}
Therefore, by considering the $G$-invariant $(1,1)$-form $\Omega_{\alpha} \in \Omega^{1,1}(X_{P})^{G}$, which satisfies $\pi^{\ast}\Omega_{\alpha} = \sqrt{-1}\partial \overline{\partial} \varphi_{\varpi_{\alpha}}$, where $\pi \colon G^{\mathbbm{C}} \to G^{\mathbbm{C}} / P = X_{P}$, and $\varphi_{\varpi_{\alpha}}(g) = \frac{1}{2\pi}\log||gv_{\varpi_{\alpha}}^{+}||^{2}$, $\forall g \in G^{\mathbbm{C}}$, we have 
\begin{equation}
\Omega_{\alpha} |_{U_{i}} = (\pi \circ s_{i})^{\ast}\Omega_{\alpha} = \frac{\sqrt{-1}}{2\pi}F_{\nabla} \Big |_{U_{i}},
\end{equation}
i.e., $c_{1}(\mathscr{O}_{\alpha}(1)) = [ \Omega_{\alpha}]$, $\forall \alpha \in \Sigma \backslash \Theta$. By considering ${\text{Pic}}(X_{P}) = H^{1}(X_{P},\mathcal{O}_{X_{P}}^{\ast})$, from the ideas described above we have the following result.
\begin{proposition}
\label{C8S8.2Sub8.2.3P8.2.6}
Let $X_{P}$ be a complex flag variety associated to some parabolic Lie subgroup $P = P_{\Theta}$. Then, we have
\begin{equation}
\label{picardeq}
{\text{Pic}}(X_{P}) = H^{1,1}(X_{P},\mathbbm{Z}) = H^{2}(X_{P},\mathbbm{Z}) = \displaystyle \bigoplus_{\alpha \in \Sigma \backslash \Theta}\mathbbm{Z}[\Omega_{\alpha} ].
\end{equation}
\end{proposition}
\begin{proof}

Let us sketch the proof. The last equality on the right-hand side of Eq. (\ref{picardeq}) follows from the following facts:

\begin{itemize}

\item[(i)] $\pi_{2}(X_{P}) \cong \pi_{1}(T(\Sigma \backslash \Theta)^{\mathbbm{C}}) = \mathbbm{Z}^{|\Sigma \backslash \Theta|}$, where $T(\Sigma \backslash \Theta)^{\mathbbm{C}}$ is given as in Remark \ref{parabolicdec};

\item[(ii)] Since $X_{P}$ is simply connected, it follows that $H_{2}(X_{P},\mathbbm{Z}) \cong \pi_{2}(X_{P})$ (Hurewicz's theorem);

\item[(iii)] By taking $\mathbbm{P}_{\alpha}^{1} \hookrightarrow X_{P}$, such that 
\begin{equation}
\label{Scurve}
\mathbbm{P}_{\alpha}^{1} = \overline{\exp(\mathfrak{g}_{-\alpha})x_{0}} \subset X_{P},
\end{equation}
for all $\alpha \in \Sigma \backslash \Theta$, where $x_{0} = eP \in X_{P}$, it follows that 
\begin{equation}
\big \langle c_{1}(\mathscr{O}_{\alpha}(1)), [ \mathbbm{P}_{\beta}^{1}] \big \rangle = \displaystyle \int_{\mathbbm{P}_{\beta}^{1}} c_{1}(\mathscr{O}_{\alpha}(1)) = \delta_{\alpha \beta},
\end{equation}
for every $\alpha,\beta \in \Sigma \backslash \Theta$. Hence, we obtain
\begin{center}

$\pi_{2}(X_{P}) = \displaystyle \bigoplus_{\alpha \in \Sigma \backslash \Theta} \mathbbm{Z} [ \mathbbm{P}_{\alpha}^{1}],$ \ \ and \ \ $H^{2}(X_{P},\mathbbm{Z}) = \displaystyle \bigoplus_{\alpha \in \Sigma \backslash \Theta}  \mathbbm{Z} c_{1}(\mathscr{O}_{\alpha}(1))$.

\end{center}
\end{itemize}
Moreover, from above we also have $H^{1,1}(X_{P},\mathbbm{Z}) = H^{2}(X_{P},\mathbbm{Z})$. In order to conclude the proof, from the Lefschetz theorem on (1,1)-classes \cite{MR2093043}, and from the fact that ${\text{rk}}({\text{Pic}}^{0}(X_{P})) = 0$, we obtain the first equality in Eq. (\ref{picardeq}).
\end{proof}

\begin{remark}[Harmonic 2-forms on $X_{P}$]Given any $G$-invariant Riemannian metric $g$ on $X_{P}$, let us denote by $\mathscr{H}^{2}(X_{P},g)$ the space of real harmonic 2-forms on $X_{P}$ with respect to $g$, and by $\mathscr{I}_{G}^{1,1}(X_{P})$ the space of closed invariant real $(1,1)$-forms. Combining the result of Proposition \ref{C8S8.2Sub8.2.3P8.2.6} with \cite[Lemma 3.1]{MR528871}, we obtain 
\begin{equation}
\mathscr{I}_{G}^{1,1}(X_{P}) = \mathscr{H}^{2}(X_{P},g). 
\end{equation}
Therefore, the closed $G$-invariant real $(1,1)$-forms described in Theorem \ref{AZADBISWAS} are harmonic with respect to any $G$-invariant Riemannian metric on $X_{P}$.
\end{remark}

\begin{remark}[K\"{a}hler cone of $X_{P}$]It follows from Eq. (\ref{picardeq}) and Theorem \ref{AZADBISWAS} that the K\"{a}hler cone of a complex flag variety $X_{P}$ is given explicitly by
\begin{equation}
\mathcal{K}(X_{P}) = \displaystyle \bigoplus_{\alpha \in \Sigma \backslash \Theta} \mathbbm{R}^{+}[ \Omega_{\alpha}].
\end{equation}
\end{remark}

\begin{remark}[Cone of curves of $X_{P}$] It is worth observing that the cone of curves ${\rm{NE}}(X_{P})$ of a flag variety $X_{P}$ is generated by the rational curves $[\mathbbm{P}_{\alpha}^{1}] \in \pi_{2}(X_{P})$, $\alpha \in \Sigma \backslash \Theta$, see for instance \cite[\S 18.3]{Timashev} and references therein.

\end{remark}

\subsection{The first Chern class of flag varieties} In this subsection, we will review some basic facts related to the Ricci form of $G$-invariant K\"{a}hler metrics on flag varieties. 

Let $X_{P}$ be a complex flag variety associated to some parabolic Lie subgroup $P = P_{\Theta} \subset G^{\mathbbm{C}}$. By considering the identification $T_{x_{0}}^{1,0}X_{P} \cong \mathfrak{m} \subset \mathfrak{g}^{\mathbbm{C}}$, such that 

\begin{center}
$\mathfrak{m} = \displaystyle \sum_{\alpha \in \Pi^{+} \backslash \langle \Theta \rangle^{+}} \mathfrak{g}_{-\alpha}$,
\end{center}
where $x_{0} = eP \in X_{P}$, we can realize $T^{1,0}X_{P}$ as being a holomoprphic vector bundle associated to the $P$-principal bundle $P \hookrightarrow G^{\mathbbm{C}} \to X_{P}$ given by

\begin{center}

$T^{1,0}X_{P} = G^{\mathbbm{C}} \times_{P} \mathfrak{m}$.

\end{center}
The twisted product on the right-hand side above is obtained from the isotropy representation ${\rm{Ad}} \colon P \to {\rm{GL}}(\mathfrak{m})$. From this, a straightforward computation shows us that 
\begin{equation}
\label{canonicalbundleflag}
K_{X_{P}}^{-1} = \det \big(T^{1,0}X_{P} \big) =\det \big ( G^{\mathbbm{C}} \times_{P} \mathfrak{m} \big )= L_{\chi_{\delta_{P}}},
\end{equation}
where $\det({\rm{Ad}}(g)) = \chi_{\delta_{P}}^{-1}(g)$, $\forall g \in P$, so $\det \circ {\rm{Ad}} = \chi_{\delta_{P}}^{-1}$. Hence, from the previous results we have 
\begin{equation}
\label{charactercanonical}
\chi_{\delta_{P}} = \displaystyle \prod_{\alpha \in \Sigma \backslash \Theta} \chi_{\varpi_{\alpha}}^{\langle \delta_{P},h_{\alpha}^{\vee} \rangle} \Longrightarrow \det \big(T^{1,0}X_{P} \big) = \bigotimes_{\alpha \in \Sigma \backslash \Theta}\mathscr{O}_{\alpha}(\ell_{\alpha}),
\end{equation}
such that $\ell_{\alpha} = \langle \delta_{P}, h_{\alpha}^{\vee} \rangle, \forall \alpha \in \Sigma \backslash \Theta$. In the above computation, notice that 
\begin{equation}
\delta_{P} = \sum_{\alpha \in \Pi^{+} \backslash \langle \Theta \rangle^{+} } \alpha.
\end{equation}
If we consider the invariant K\"{a}hler metric $\rho_{0} \in \Omega^{1,1}(X_{P})^{G}$, locally describe by
\begin{equation}
\label{riccinorm}
\rho_{0}|_{U} = \sum_{\alpha \in \Sigma \backslash \Theta}\langle \delta_{P}, h_{\alpha}^{\vee} \rangle \sqrt{-1} \partial \overline{\partial}\log \big (||s_{U}v_{\varpi_{\alpha}}^{+}||^{2}\big ),
\end{equation}
for some local section $s_{U} \colon U \subset X_{P} \to G^{\mathbbm{C}}$, it is straightforward to see that 
\begin{equation}
\label{ChernFlag}
c_{1}(X_{P}) = \Big [ \frac{\rho_{0}}{2\pi}\Big].
\end{equation}
By the uniqueness of $G$-invariant representative of $c_{1}(X_{P})$, it follows that 
\begin{center}
${\rm{Ric}}(\rho_{0}) = \rho_{0}$, 
\end{center}
i.e., $\rho_{0} \in \Omega^{1,1}(X_{P})^{G}$ defines a $G$-ivariant K\"{a}hler-Einstein metric on $X_{P}$ (cf. \cite{MATSUSHIMA}). 
\begin{remark}
\label{scalarcurvature}
From the uniqueness of the $G$-invariant representative for $c_{1}(X_{P})$, given any $G$-invariant K\"{a}hler metric $\omega_{\varphi}$, we have that ${\rm{Ric}}(\omega_{\varphi}) = \rho_{0}$. Therefore, the scalar curvature $S(\omega_{\varphi})$ of $\omega_{\varphi}$ is given by
\begin{equation}
S(\omega_{\varphi}) = \Lambda_{\omega_{\varphi}}({\rm{Ric}}(\omega_{\varphi})) = {\rm{tr}}_{\omega_{\varphi}}(\rho_{0}).
\end{equation}
Since $\rho_{0}$ is harmonic with respect to any $G$-invariant K\"{a}hler metric, we have that $S(\omega_{\varphi})$ is constant. 
\end{remark}

Given any two $G$-invariant K\"{a}hler metrics $\omega_{1}$ and $\omega_{2}$ on $X_{P}$, we have ${\rm{Ric}}(\omega_{1}) = {\rm{Ric}}(\omega_{2}) = \rho_{0}$. Thus, it follows that the smooth function $\frac{\det(\omega_{1})}{\det(\omega_{2})}$ is constant. Moreover, we have
\begin{equation}
{\rm{Vol}}(X_{P},\omega_{1})  = \frac{\det(\omega_{1})}{\det(\omega_{2})}{\rm{Vol}}(X_{P},\omega_{2}). 
\end{equation}
In particular, denoting $V_{0} = {\rm{deg}}(X_{P},-K_{X_{P}}) = (-K_{X_{P}})^{n}$, we can show from above the following result.
\begin{theorem}[Azad-Biswas, \cite{AZAD}]
\label{volumeflagform}
The volume of $X_{P}$ with respect to an arbitrary $G$-invariant K\"{a}hler metric $\omega = \sum_{\alpha \in \Sigma \backslash \Theta}c_{\alpha} \Omega_{\alpha}$, such that $c_{\alpha} > 0$, $\forall \alpha \in \Sigma \backslash \Theta$, is given by 
\begin{equation}
\label{volumeflag}
{\rm{Vol}}(X_{P},\omega) = \frac{V_{0}}{n!}\frac{\prod_{\gamma \in \Pi^{+} \backslash \langle \Theta \rangle^{+}}\Big [\sum_{\alpha \in \Sigma \backslash \Theta}c_{\alpha} \langle \varpi_{\alpha}, h_{\gamma}^{\vee} \rangle \Big ] }{\prod_{\gamma \in \Pi^{+} \backslash \langle \Theta \rangle^{+}} \Big [\sum_{\alpha \in \Sigma \backslash \Theta}\langle \delta_{P}, h_{\alpha}^{\vee} \rangle \langle \varpi_{\alpha}, h_{\gamma}^{\vee} \rangle \Big ]}.
\end{equation}
\end{theorem}

\begin{remark}
It is worth pointing out that the expression given in Eq. (\ref{volumeflag}) is slightly different from \cite{AZAD}. The reason for this is that we consider the volume of $X_{P}$ with respect to an arbitrary K\"{a}hler metric $\omega$ as being $\frac{1}{n!} \int_{X_{P}}\omega^{n}$, instead of $\int_{X_{P}}\omega^{n}$. Given a K\"{a}hler class $\xi \in \mathcal{K}(X_{P})$, we define the volume of $\xi$ as being
\begin{equation}
{\rm{Vol}}(\xi) := n! {\rm{Vol}}(X_{P},\omega), 
\end{equation}
for some $\omega \in \xi$. Thus, according to our convention, the formula presented in \cite{AZAD} for the volume of $X_{P}$ with respect to an arbitrary $G$-invariant K\"{a}hler metric $\omega$ corresponds to ${\rm{Vol}}([\omega])$.
\end{remark}

\section{Proof of main results}\label{main-results}

\begin{theorem}\label{main-thm}
Let $L \in {\rm{Pic}}(X_{P})$ and let $\beta \in c_{1}(L)$ be a $G$-invariant $(1,1)$-form. Then there exists a (unique) $G$-invariant K\"{a}hler metric $\omega$ on $X_{P}$, satisfying 
\begin{equation}
\label{TKE}
{\rm{Ric}}(\omega) = \omega + \beta,
\end{equation}
if, and only if,
\begin{equation}
\int_{\mathbbm{P}_{\alpha}^{1}}\beta < 2\pi\langle \delta_{P},h_{\alpha}^{\vee}  \rangle,
\end{equation}
$\forall \alpha \in \Sigma \backslash \Theta$, where $\mathbbm{P}_{\alpha}^{1} \subset X_{P}$, $\alpha \in \Sigma \backslash \Theta$, are generators of the cone of curves ${\rm{NE}}(X_{P})$. 
\end{theorem}

\begin{proof}
Given $L \in {\rm{Pic}}(X_{P})$, it follows that 
\begin{equation}
L = \bigotimes_{\alpha \in \Sigma \backslash \Theta} \mathscr{O}_{\alpha}(\ell_{\alpha}),
\end{equation}
such that $\ell_{\alpha} \in \mathbbm{Z}$, $\forall \alpha \in \Sigma \backslash \Theta$. If $\omega$ is a $G$-invariant K\"{a}hler metric on $X_{P}$ satisfying Eq. (\ref{TKE}), it follows that 
\begin{equation}
2\pi c_{1}(X_{P}) = [\omega] + c_{1}(L) .
\end{equation}
Hence, we have 
\begin{equation}
[\omega] = \sum_{\alpha \in \Sigma \backslash \Theta}\big ( 2 \pi \langle \delta_{P},h_{\alpha}^{\vee}  \rangle - \ell_{\alpha}\big )c_{1}(\mathscr{O}_{\alpha}(1)).
\end{equation}
Since $\omega$ is a positive real $(1,1)$-form, it follows that $2 \pi \langle \delta_{P},h_{\alpha}^{\vee}  \rangle - \ell_{\alpha} > 0$, $\forall \alpha \in \Sigma \backslash \Theta$. Thus, we conclude that 
\begin{equation}
\int_{\mathbbm{P}_{\alpha}^{1}}\beta = \ell_{\alpha} < 2 \pi \langle \delta_{P},h_{\alpha}^{\vee}  \rangle, 
\end{equation}
for all $\alpha \in \Sigma \backslash \Theta$. On the other hand, given a $G$-invariant $(1,1)$-form $\beta \in c_{1}(L)$, such that $\int_{\mathbbm{P}_{\alpha}^{1}}\beta < 2\pi\langle \delta_{P},h_{\alpha}^{\vee}  \rangle$, $\forall \alpha \in \Sigma \backslash \Theta$, we set 
\begin{equation}
\omega := \sum_{\alpha \in \Sigma \backslash \Theta}\bigg ( 2 \pi \langle \delta_{P},h_{\alpha}^{\vee}  \rangle - \int_{\mathbbm{P}_{\alpha}^{1}}\beta \bigg ) \Omega_{\alpha}.
\end{equation}
By the above definition, we have that $\omega$ defines a $G$-invariant K\"{a}hler metric on $X_{P}$. Moreover, it is straightforward to verify that ${\rm{Ric}}(\omega) = \omega + \beta$. The uniqueness of $\omega$ follows from the fact that it is $G$-invariant. 
\end{proof}

In the setting of the theorem above, if we replace $c_{1}(L)$ by $2\pi \big (c_{1}(X_{P}) - \xi)$, for some K\"{a}hler class $\xi$ on $X_{P}$, observing that every $\beta \in 2\pi \big (c_{1}(X_{P}) - \xi)$ satisfies
\begin{equation}
\int_{\mathbbm{P}_{\alpha}^{1}}\beta = 2\pi \big \langle c_{1}(X_{P}) - \xi,\mathbbm{P}_{\alpha}^{1} \big \rangle < 2\pi \int_{\mathbbm{P}_{\alpha}^{1}}c_{1}(X_{P}) = 2 \pi \langle \delta_{P},h_{\alpha}^{\vee}  \rangle,
\end{equation}
for all $\alpha \in \Sigma \backslash \Theta$, from Theorem \ref{main-thm} we obtain the following corollary.

\begin{corollary}
\label{CorollarytwistedKE}
Given a K\"{a}hler class $\xi$ on $X_{P}$ and a $G$-invariant $(1,1)$-form $ \beta \in 2\pi \big ( c_{1}(X_{P}) - \xi \big)$, then there exist a unique $G$-invariant K\"{a}hler metric $\omega \in 2\pi \xi$, such that 
\begin{equation}
\label{twistedKE}
{\rm{Ric}}(\omega) = \omega + \beta.
\end{equation}
\end{corollary}

By taking the trace with respect to $\omega$ in Eq. (\ref{twistedKE}), we obtain the following result.

\begin{corollary}
Given a K\"{a}hler class $\xi$ on $X_{P}$ and a $G$-invariant $(1,1)$-form $ \beta \in 2\pi \big ( c_{1}(X_{P}) - \xi \big)$, then there exists a (unique) $G$-invariant K\"{a}hler metric $\omega \in 2\pi \xi$ with constant $\beta$-twisted scalar curvature, such that
\begin{equation}
S(\omega) - \Lambda_{\omega}(\beta) = \dim_{\mathbbm{C}}(X_{P}),
\end{equation}
where $S(\omega)$ denotes the Chern scalar curvature of $\omega$.
\end{corollary}

From the Corollary \ref{CorollarytwistedKE} we can show the following.

\begin{corollary}
Let $R \colon \mathcal{K}(X_{P}) \to \mathbbm{R}$, such that $R(\xi)$ is the greatest Ricci lower bound of $\xi \in \mathcal{K}(X_{P})$. Then, we have
\begin{equation}
R(\xi) = \min \bigg \{ \frac{\langle \delta_{P},h_{\alpha}^{\vee}  \rangle}{a_{\alpha}} \ \bigg | \  \alpha \in \Sigma \backslash \Theta \bigg \},
\end{equation}
such that $a_{\alpha} = \langle \xi, [\mathbbm{P}_{\alpha}^{1}]\rangle$, $\forall \alpha \in \Sigma \backslash \Theta$.
\end{corollary}

\begin{proof}
Given $\xi \in \mathcal{K}(X_{P})$, we have $\xi = \sum_{\alpha \in  \Sigma \backslash \Theta}a_{\alpha}[\Omega_{\alpha}]$, with $a_{\alpha} > 0$, $\forall \alpha \in  \Sigma \backslash \Theta$. Let us denote 
\begin{center}
$\displaystyle R_{0} = \min \bigg \{ \frac{\langle \delta_{P},h_{\alpha}^{\vee}  \rangle}{a_{\alpha}} \ \bigg | \  \alpha \in \Sigma \backslash \Theta \bigg \}$.
\end{center}
For any $\epsilon > 0$, we can always choose $s > 0$, such that $R_{0} - \epsilon < s \leq R_{0}$. By taking a $G$-invariant $(1,1)$-form $ \beta \in 2\pi \big ( c_{1}(X_{P}) - s\xi \big)$, it follows from Corollary \ref{CorollarytwistedKE} that there exist a unique $G$-invariant tKE metric $\omega \in 2\pi s \xi$. From this, we have 
\begin{center}
${\rm{Ric}}(\frac{\omega}{s}) = {\rm{Ric}}(\omega) = \omega + \beta = s \big (\frac{\omega}{s} \big) + \beta \Rightarrow {\rm{Ric}}(\frac{\omega}{s}) - s \big (\frac{\omega}{s} \big) = \beta$.
\end{center}
Since $s \leq R_{0} \Rightarrow a_{\alpha}s \leq \langle \delta_{P},h_{\alpha}^{\vee}  \rangle$, $\forall \alpha \in  \Sigma \backslash \Theta$, we have
\begin{center}
$\displaystyle c_{1}(X_{P}) - s\xi = \sum_{\alpha \in  \Sigma \backslash \Theta}(\langle \delta_{P},h_{\alpha}^{\vee}  \rangle - sa_{\alpha})[\Omega_{\alpha}] \geq 0$,
\end{center}
thus ${\rm{Ric}}(\frac{\omega}{s}) - s \big (\frac{\omega}{s} \big) = \beta \geq 0$, with $\frac{\omega}{s} \in 2\pi \xi$. From this, we obtain
\begin{center}
$s \in \{ r \in \mathbbm{R} \ | \ \exists \ {\text{K\"{a}hler form}} \ \omega \in 2\pi \xi, \ {\text{s.t.}} \ {\rm{Ric}}(\omega) \geq r \omega  \}$.
\end{center}
Therefore, by definition of $R(\xi)$, we conclude that $R_{0} = R(\xi)$.
\end{proof}

From the description provided above for the greatest Ricci lower bound, we can prove the following result using essentially tools from Lie theory.

\begin{theorem}
For every $\xi \in \mathcal{K}(X_{P})$, the following inequalities hold
\begin{equation}
R(\xi)^{n}{\rm{Vol}}(\xi) \leq (-K_{X_{P}})^{n} \leq (n+1)^{n},
\end{equation}
such that $n = \dim_{\mathbbm{C}}(X_{P})$.
\end{theorem}

\begin{proof}
Given $\xi \in \mathcal{K}(X_{P})$, it follows that $\xi = \sum_{\alpha \in  \Sigma \backslash \Theta}a_{\alpha}[\Omega_{\alpha}]$, with $a_{\alpha} > 0$, $\forall \alpha \in  \Sigma \backslash \Theta$. By considering $\omega =  \sum_{\alpha \in  \Sigma \backslash \Theta}a_{\alpha}\Omega_{\alpha}$, we obtain from Theorem \ref{volumeflag} the following
\begin{center}
${\rm{Vol}}(\xi) = n!{\rm{Vol}}(X_{P},\omega) = \displaystyle {V_{0}\frac{\prod_{\gamma \in \Pi^{+} \backslash \langle \Theta \rangle^{+}}\Big [\sum_{\alpha \in \Sigma \backslash \Theta}a_{\alpha} \langle \varpi_{\alpha}, h_{\gamma}^{\vee} \rangle \Big ] }{\prod_{\gamma \in \Pi^{+} \backslash \langle \Theta \rangle^{+}} \Big [\sum_{\alpha \in \Sigma \backslash \Theta}\langle \delta_{P}, h_{\alpha}^{\vee} \rangle \langle \varpi_{\alpha}, h_{\gamma}^{\vee} \rangle \Big ]}}.$
\end{center}
Since $R(\xi) \leq \frac{\langle \delta_{P},h_{\alpha}^{\vee}  \rangle}{a_{\alpha}}, \forall \alpha \in \Sigma \backslash \Theta$, it follows that $a_{\alpha} R(\xi) \leq \langle \delta_{P},h_{\alpha}^{\vee}  \rangle$, $\forall \alpha \in \Sigma \backslash \Theta$. Therefore, we obtain
\begin{equation}
R(\xi)^{n}{\rm{Vol}}(\xi) = \displaystyle {V_{0}\frac{\prod_{\gamma \in \Pi^{+} \backslash \langle \Theta \rangle^{+}}\Big [\sum_{\alpha \in \Sigma \backslash \Theta}a_{\alpha}R(\xi) \langle \varpi_{\alpha}, h_{\gamma}^{\vee} \rangle \Big ] }{\prod_{\gamma \in \Pi^{+} \backslash \langle \Theta \rangle^{+}} \Big [\sum_{\alpha \in \Sigma \backslash \Theta}\langle \delta_{P}, h_{\alpha}^{\vee} \rangle \langle \varpi_{\alpha}, h_{\gamma}^{\vee} \rangle \Big ]}} \leq V_{0} = {\rm{deg}}(X_{P},-K_{X_{P}}),
\end{equation}
notice that $n = \dim_{\mathbbm{C}}(X_{P}) = |\Pi^{+} \backslash \langle \Theta \rangle^{+}|$. Following \cite[Theorem 24.10]{BorelHizebruch} and \cite[Example 18.13]{Timashev}, we have 
\begin{equation}
{\rm{deg}}(X_{P},-K_{X_{P}}) = (-K_{X_{P}})^{n} = \int_{X_{P}}c_{1}(X_{P})^{n} = n!\prod_{\gamma \in \Pi^{+} \backslash \langle \Theta \rangle^{+}}\frac{ \langle \delta_{P},h_{\gamma}^{\vee} \rangle}{ \langle \varrho^{+},h_{\gamma}^{\vee} \rangle}.
\end{equation}
Therefore, from the result provided in \cite[Theorem 1]{Snow} we obtain the desired inequalities.
\end{proof}

\section{Examples}\label{sec-examples}

In order to compute the numbers $\langle \delta_{P},h_{\alpha}^{\vee} \rangle$, $\alpha \in \Sigma \backslash \Theta$, which appear in the formulas of our results, it is worth recalling that $\delta_P=c_1\varpi_{\alpha_1} +\ldots + c_n\varpi_{\alpha_n}$, where $c_i>0$ and $\alpha_i \in \Sigma \backslash \Theta$ (cf. Borel-Hirzebruch \cite{Bo-Hi}). The numbers $c_i$ are called {\em Koszul numbers}, these numbers characterize the K\"ahler-Einstein metrics on flag varieties (see \cite{kimura}, \cite{AC1}, \cite{AC2}). In what follows, we present some examples which illustrate how the existence of $G$-invariant tKE metrics can be characterized using Koszul numbers. 

\subsection{Full flag variety $X_{B} = G^{\mathbbm{C}}/B$} In the setting of Theorem \ref{main-thm}, in the particular case that $P = B$ (i.e. $\Theta  = \emptyset$), it follows that $\delta_{B} = 2 \varrho^{+}$, where
\begin{equation}
\varrho^{+} = \frac{1}{2}\sum_{\alpha \in \Pi^{+}}\alpha = \sum_{\alpha \in \Sigma} \varpi_{\alpha}.
\end{equation}
Therefore, considering a $G$-invariant $(1,1)$-form $\beta \in c_{1}(L)$, for some $L \in {\rm{Pic}}(X_{B})$, from Theorem \ref{main-thm} we have that the necessary and sufficient condition over $\beta$ for the existence of a $G$-invariant K\"{a}hler metric $\omega$ on $X_{B}$, satisfying ${\rm{Ric}}(\omega) = \omega + \beta$, is given by
\begin{equation}
\int_{\mathbbm{P}_{\alpha}^{1}}\beta  < 2\pi \langle \delta_{B},h_{\alpha}^{\vee}  \rangle = 2\pi \langle 2\varrho^{+},h_{\alpha}^{\vee}  \rangle = 4\pi \langle  \sum_{\gamma \in \Sigma} \varpi_{\gamma},h_{\alpha}^{\vee} \rangle = 4 \pi, \ \ \forall \alpha \in \Sigma.
\end{equation} 
\subsection{Projetivization of $T\mathbbm{P}^{n+1}$}
Let us consider the flag variety $X_P= {\rm{SU}}(n+2)/{\rm{S}}({\rm{U}}(n)\times {\rm{U}}(1)\times {\rm{U}}(1))$. This space can be viewed as the projetivization of the tangent bundle of $\mathbbm{P}^{n+1}$, so let us denote $X_{P} = \mathbbm{P}(T\mathbbm{P}^{n+1})$, see \cite{hirzen}. Let $\mathfrak{g}^{\mathbbm{C}}$ be the Lie algebra $\mathfrak{sl}(n+2, \mathbbm{C})$ with Cartan subalgebra given by the diagonal traceless matrices. Recall that the roots of the Lie algebra $\mathfrak{g}^{\mathbbm{C}} = \mathfrak{sl}(n+2,\mathbbm{C})$ is given by the functionals $\alpha_{ij}=\lambda_i -\lambda_j$, where $\lambda_i(\Diag(a_1, \ldots, a_{n+2}))=a_i$, $i = 1,\ldots,n+2$. The set of simple roots $\Sigma$ is given $\alpha_{i,i+1}$, $i = 1,\ldots,n+1$. The space $\mathbbm{P}(T\mathbbm{P}^{n+1})$ is characterized by $P = P_{\Theta}$, such that $\Sigma\setminus \Theta=\{ \alpha_{n,n+1}, \alpha_{n+1,n+2} \}$. In particular, we have $\rank  H^{2}(\mathbbm{P}(T\mathbbm{P}^{n+1}),\mathbbm{Z})=2$. A straightforward computation gives in this case
$$
\delta_P=2\lambda_1 +\ldots + 2\lambda_n - (n-1)\lambda_{n+1} - (n+1)\lambda_{n+2}, 
$$
thus
\begin{equation}
    \langle \delta_P,h_{\alpha_{n,n+1}}^{\vee}  \rangle = n+1, \qquad   \langle \delta_P,h_{\alpha_{n+1,n+2}}^{\vee}  \rangle = 2.
\end{equation}
Therefore, considering a ${\rm{SU}}(n+2)$-invariant $(1,1)$-form $\beta \in c_{1}(L)$, for some $L \in {\rm{Pic}}(\mathbbm{P}(T\mathbbm{P}^{n+1}))$, then there exists a ${\rm{SU}}(n+2)$-invariant K\"{a}hler metric $\omega$ on $\mathbbm{P}(T\mathbbm{P}^{n+1})$, such that ${\rm{Ric}}(\omega) = \omega + \beta$ if, and only if, 
\begin{equation}
\int_{\mathbbm{P}_{\alpha_{n,n+1}}^{1}}\beta  < 2\pi \langle \delta_{P},h_{\alpha_{n,n+1}}^{\vee} \rangle = 2\pi (n+1) \ \ \text{and} \ \ \int_{\mathbbm{P}_{\alpha_{n+1,n+2}}^{1}}\beta  < 2\pi \langle \delta_{P},h_{\alpha_{n+1,n+2}}^{\vee} \rangle = 4\pi. 
\end{equation}
\subsection{Flag varieties with Picard number two and few isotropy summands} According to Proposition \ref{C8S8.2Sub8.2.3P8.2.6}, flag varieties with Picard number 2 are parameterized by $\Theta \subset \Sigma$, such that $\Sigma\setminus\Theta=\{\alpha, \gamma \}$. In what follows, we will characterize tKE metrics for some families of such spaces parameterized in terms of the number of components of the isotropy representation as follows (see \cite{kimura}, \cite{AC1}, \cite{AC2}): let $\mu=n_1\alpha_1 + \ldots + n_m\alpha_m$ be the maximal root of $\Pi$ and recall the {\em height} of a simple root $\alpha_i$ with respect to $\mu$ is the positive number $\htt (\alpha_i)=n_i$. In this setting, we have the following classification: 
\begin{enumerate}
\item {\bf Type I:} flags with three isotropy summands: $\Sigma\setminus\Theta=\{\alpha, \gamma : \, \htt(\alpha)=\htt(\gamma)=1 \}$ 
\item {\bf Type II:} flags with four isotropy summands: $\Sigma\setminus\Theta=\{\alpha, \gamma : \, \htt(\alpha)= 1, \, \htt(\gamma)=2 \}$ 
\item {\bf Type III:} flags with five isotropy summands: $\Sigma\setminus\Theta=\{\alpha, \gamma : \, \htt(\alpha)= 1, \, \htt(\gamma)=2  \mbox{ \em{or} } \htt(\alpha)= 2, \, \htt(\gamma)=2 \}$ 
\end{enumerate}

      
 
    
   

The next proposition summarize the classification of flag varieties of type I, II, III. 

\begin{proposition}[\cite{AC1},\cite{AC2},\cite{kimura}]\label{flag-type} The flag varieties of type I, II and III (up to equivalence) and the corresponding Koszul numbers are listed in the Table \ref{tab:my_label}. The black dots on the Dynkin diagram represents the simple roots on $\Sigma\setminus\Theta$ and the labels on the Dynkin diagram denotes the {\em height} of the simple root with respect to the maximal root $\mu$.

\begin{table}[h!]

    \centering
   
    \begin{tabular}{c|c|c|c}
    
        Type & $X_{P} = G/ G \cap P$ & $\Sigma\setminus\Theta = \{ \alpha_1,\alpha_2\}$ &$\delta_P=\ell_1\varpi_{\alpha_1}+\ell_2\varpi_{\alpha_2}$ \\ \hline \hline
        
        \multirow{2}{*}{I} & $SO(2\ell)/U(1)\times U(\ell-1)$ & \multirow{2}{*}{\dynkin[labels={1,2,2,2,2,1,1},scale=1.4]D{ooo...oo**}} & \multirow{2}{*}{$\ell \, \varpi_{\alpha_1} + \ell \, \varpi_{\alpha_2}$}   \\
        &$(\ell \geq 4)$& &  \\ & & &  \\ \hline

        \multirow{2}{*}{I} &  $SO(2\ell)/U(1)\times U(\ell-1)$ & \multirow{2}{*}{\dynkin[labels={1,2,2,2,2,1,1},scale=1.4]D{*oo...oo*o}} & \multirow{2}{*}{$\ell \, \varpi_{\alpha_1} + 2(\ell-2) \, \varpi_{\alpha_2}$}   \\ 
        & $(\ell \geq 4)$ & &  \\ & & & \\ \hline
        
        \multirow{2}{*}{I} & $SO(2\ell)/U(1)\times U(\ell-1)$ & \multirow{2}{*}{\dynkin[labels={1,2,2,2,2,1,1},scale=1.4]D{*oo...ooo*}} & \multirow{2}{*}{$\ell \, \varpi_{\alpha_1} + 2(\ell-2) \, \varpi_{\alpha_2}$}   \\ 
        &$(\ell \geq 4)$ & &  \\  & & & \\ \hline
         
          \multirow{2}{*}{I} & $SU(\ell +n+m)/SU(U(\ell)\times U(m) \times U(n))$ & \multirow{2}{*}{\dynkin[labels={1,1,1,1,1,1,1},scale=1.4]A{o*o...oo*o}} & \multirow{2}{*}{$(\ell+m) \, \varpi_{\alpha_1} + (m+n) \, \varpi_{\alpha_2}$ } \\ 
          &$(\ell, m, n \geq 1 )$ &  &  \\ \hline 
          
          I & $E_6/U(1)\times U(1) \times \Spin(8)$ & \dynkin[labels={1,2,2,3,2,1},scale=1.4]E{*oooo*} & $4 \, \varpi_{\alpha_1} + 4 \, \varpi_{\alpha_2}$  \\ \hline 
          
           II & $SO(2\ell+1)/SO(2\ell-3) \times U(1) \times U(1)$ & \dynkin[labels={1,2,2,2,2,2},scale=1.4]B{**o...ooo} & $ 2\varpi_{\alpha_1} + (2\ell -3) \, \varpi_{\alpha_2}$  \\ \hline 
           
            \multirow{2}{*}{II} & $Sp(\ell)/U(p)\times U(\ell -p)$ &  \multirow{2}{*}{\dynkin[labels={2,2,2,2,2,1},scale=1.4]C{ooo...*...o*}} & \multirow{2}{*}{$ \ell \, \varpi_{\alpha_1} + (\ell -p+1) \, \varpi_{\alpha_2}$}   \\ 
            & $(1\leq p \leq \ell-1)$ & &  \\ \hline 
            
             II & $SO(2\ell)/SO(2(\ell-2))\times U(1) \times U(1)$ &   \dynkin[labels={1,2,2,2,1,1},scale=1.4]D{**o...ooo} & $ 2\ell \, \varpi_{\alpha_1} + 2(\ell -2) \, \varpi_{\alpha_2}$  \\ \hline 
             
              \multirow{2}{*}{II} & $SO(2\ell)/U(p) \times U(\ell-p)$ &    \multirow{2}{*}{\dynkin[labels={1,2,2,2,1,1},scale=1.4]D{*o...*...ooo}} &  \multirow{2}{*}{$ \ell \, \varpi_{\alpha_1} + 2(\ell -p-1) \, \varpi_{\alpha_2}$}  \\
               & $(2\leq p \leq \ell-2)$ & &  \\ & & & \\ \hline 
             
             II & $ E_6/SU(5)\times U(1) \times U(1)$  &   \dynkin[labels={1,2,2,3,2,1},scale=1.4]E{*o*ooo} & $ 2 \, \varpi_{\alpha_1} + 8 \, \varpi_{\alpha_2}$  \\ \hline 
             
             II & $ E_7/SO(10)\times U(1) \times U(1)$ &   \dynkin[labels={1,2,2,3,4,3,2},scale=1.4]E{*o*oooo}& $ 2 \, \varpi_{\alpha_1} + 12 \, \varpi_{\alpha_2}$  \\ \hline 
             
              \multirow{2}{*}{III} & $ SO(2\ell +1)/U(1)\times U(p) \times SO(2(\ell-p-1)+1)$ &   \multirow{2}{*}{\dynkin[labels={1,2,2,2,2,2},scale=1.4]B{*o...*...ooo}}& \multirow{2}{*}{$ (p+1) \, \varpi_{\alpha_1} + (2\ell-p-2) \, \varpi_{\alpha_2}$ }  \\ 
              & $(\ell \geq 5, 3\geq p \geq \ell-3)$ & & \\ \hline 
    \end{tabular}
    \caption{}
    \label{tab:my_label}
\end{table}
\end{proposition}

Applying Theorem \ref{main-thm} joint with Proposition \ref{flag-type}, one can characterize and classify tKE metrics on these families of flag varieties as follows.

\begin{corollary}
Let us consider a flag variety $X_{P}$ of type I, II, III, and $\delta_P=\ell_1\varpi_{\alpha_1}+\ell_2\varpi_{\alpha_2}$ as listed on Table \ref{tab:my_label}. Let $L \in {\rm{Pic}}(X_{P})$ and let $\beta \in c_{1}(L)$ be a $G$-invariant $(1,1)$-form, then there exist a (unique) $G$-invariant K\"{a}hler metric $\omega$ on $X_{P}$, such that ${\rm{Ric}}(\omega) = \omega + \beta$ if, and only if, 
\begin{equation}
n_1:=\int_{\mathbbm{P}_{\alpha_{1}}^{1}}\beta  <  2 \pi \, \ell_1, \ \ \ {\text{and}} \ \ \ n_2:= \int_{\mathbbm{P}_{\alpha_{2}}^{1}}\beta  < 2\pi \ell_2.
\end{equation}
In this case, the K\"ahler metric is given by $\omega= (2\pi \ell_{1} - n_1) \Omega_{\alpha_1} + (2\pi\ell_{2} - n_2) \Omega_{\alpha_2}$, where the forms $\Omega_{\alpha_i}, i=1,2$, are the generators of $H^{2}(X_{P},\mathbbm{Z})$ given in Eq. (\ref{picardeq}). 
\end{corollary}

\

{\bf Acknowledgment:} E. M. Correa is supported by FAEPEX/Unicamp grant 2528/22. L. Grama is partially supported by S\~ao Paulo Research Foundation FAPESP grants 2018/13481-0, 2021/04003-0, 2021/04065-6  and CNPq grant no. 305036/2019-0.

\bibliographystyle{alpha}
\bibliography{main}

\end{document}